\documentclass[a4paper,reqno,11pt]{amsart}
\usepackage{amsmath}
\usepackage{amssymb}
\usepackage{paralist, xspace, graphicx, url, amscd, euscript, mathrsfs,stmaryrd,epic,eepic,color}
\usepackage{cases}
\usepackage{enumitem}
\allowdisplaybreaks[4]

\newtheorem{thm}{Theorem}[section]
\newtheorem{prop}[thm]{Proposition}

\newtheorem{lem}[thm]{Lemma}

\newtheorem{cor}[thm]{Corollary}

\theoremstyle{definition}

\theoremstyle{remark}
\newtheorem{remark}[thm]{Remark}

\numberwithin{equation}{section}

\newcommand\OO{{\mathcal{O}}}

\newcommand{\mult}{\operatorname{mult}}

\newcommand{\pr}{\operatorname{pr}}
\newcommand{\vol}{\operatorname{vol}}
\newcommand{\codim}{\operatorname{codim}}

\usepackage{todonotes}

\begin{document}

\title{Algebraic reverse Khovanskii--Teissier inequality via Okounkov bodies}
\date{\today}
\author{Chen Jiang}
\address{Chen Jiang, Shanghai Center for Mathematical Sciences, Fudan University, Jiangwan Campus, Shanghai, 200438, China}
\email{chenjiang@fudan.edu.cn}
\author{Zhiyuan Li}
\address{Zhiyuan Li, Shanghai Center for Mathematical Sciences, Fudan University, Jiangwan Campus, Shanghai, 200438, China}
\email{zhiyuan\_li@fudan.edu.cn}

\begin{abstract}
Let $X$ be a projective variety of dimension $n$ over an algebraically closed field of arbitrary characteristic and 
let $A, B, C$ be nef divisors on $X$. We show that for any integer $1\leq k\leq n-1$,
$$
(B^k\cdot A^{n-k})\cdot (A^k\cdot C^{n-k})\geq \frac{k!(n-k)!}{n!}(A^n)\cdot (B^k\cdot C^{n-k}).
$$
The same inequality in the analytic setting was obtained by Lehmann and Xiao for compact K\"ahler manifolds using the Calabi--Yau theorem, while our approach is purely algebraic using (multipoint) Okounkov bodies. 
We also discuss applications of this inequality to B\'ezout-type inequalities and inequalities on degrees of dominant rational self-maps.
\end{abstract}

\maketitle
\pagestyle{myheadings} \markboth{\hfill C.~Jiang, Z.~Li \hfill}{\hfill Algebraic reverse Khovanskii--Teissier inequality via Okounkov bodies \hfill}


\section{Introduction}

In \cite{LX17}, Lehmann and Xiao have proved the so called reverse Khovanskii--Teissier inequality on compact K\"ahler manifolds. Namely, for nef $(1,1)$-classes $\alpha, \beta$ and $\gamma$ on a compact K\"ahler manifold of dimension $n$, we have 
\begin{align}\label{eq:LX}
 (\beta^k\cdot \alpha^{n-k})\cdot (\alpha^k\cdot \gamma^{n-k})\geq \frac{k!(n-k)!}{n!}(\alpha ^n)\cdot (\beta^k\cdot \gamma^{n-k}).
\end{align}
In fact such an inequality was first observed by Xiao \cite{Xia15} with a weaker constant $\frac{k!(n-k)!}{4n!}$ which can be improved to $\frac{k!(n-k)!}{n!}$ by the technique of Popovici \cite{Pop16} (cf. \cite[Remark 3.1]{Xia15}).
Such an inequality between intersection numbers plays an important role in the proof of Morse type inequalities (see also \cite{Pop16, Xia15}) and it also has interesting applications to dynamical degrees of rational maps \cite{BFJ08, Dan20}.
Besides, it has an interesting analogue in 
convex geometry \cite[Theorem~5.9]{LX17}
and applications in the study of convolution of convex valuations \cite[Theorem~2.9]{DX21}.

The proof of Popovici and Lehmann--Xiao depends on solving Monge--Amp\'ere equations which deeply relies on the Calabi--Yau theorem, and it has been asked in \cite[Remark 9.3]{LX'15} (the arXiv version of \cite{LX16}) whether there is an algebraic approach working for projective varieties defined over arbitrary fields. 
In \cite[Theorem~3.4.3]{Dan20}, by studying intersections of numerical cycles, Dang has proved a weaker form of \eqref{eq:LX} in the algebraic geometry setting with the constant $\frac{k!(n-k)!}{n!}$ replacing by $\frac{1}{(n-k+1)^k}$.

In this paper, we use the (multipoint) Okounkov bodies to prove this optimal inequality in the algebraic geometry setting. The main result is 

\begin{thm}\label{thm:main}
Let $X$ be a projective variety of dimension $n$ over an algebraically closed field of arbitrary characteristic.
Let $A, B, C$ be nef divisors on $X$. Then for any integer $1\leq k\leq n-1$,
\begin{align}
 (B^k\cdot A^{n-k})\cdot (A^k\cdot C^{n-k})\geq \frac{k!(n-k)!}{n!}(A^n)\cdot (B^k\cdot C^{n-k}).\label{eq:main}
\end{align}
Moreover, 
if $A, B, C$ are ample, then this inequality is strict.
\end{thm}
 
We shall mention that the constant $ \frac{k!(n-k)!}{n!}$ in \eqref{eq:main} is optimal. For instance, if $X=\mathbb{P}^k\times \mathbb{P}^{n-k}$, one can take $A, B, C$ to be the divisors on $X$ of type $(1, 1)$, $(1, 0)$, $(0,1)$, respectively, then the equality holds in \eqref{eq:main}. It is interesting and natural to ask for the characterization of the equality case of \eqref{eq:main}, but the answer might be very complicated. One naive guess is that if the equality holds while both sides are non-zero, 
then $B$ and $C$ have numerical dimensions $k$ and $n-k$ respectively and $A$ lies on the plane spanned by $B$ and $C$ in the N\'eron--Severi group $\mathrm{NS}(X)\otimes_\mathbb{Z} \mathbb{R}$. We will give such a characterization when $X$ is a surface in Proposition~\ref{prop:equality}.

Our proof makes use of the connection between volumes of big divisors and volumes of (multipoint) Okounkov bodies discovered in \cite{LM09, KK12} and \cite{Tru18}. The desired inequality follows from a dedicated comparison between the Okounkov bodies associated to an admissible flag and the convex bodies associated to sub-flags obtained by cutting the admissible flag with very general hyperplanes. In this way, we actually obtain a more general result for restricted volumes, see Theorem~\ref{thm:1}, Theorem~\ref{thm:2}, and Theorem~\ref{thm:1restricted} for details. 
Another advantage of this method is that we only need to deal with the intersection theory of divisors, which is much easier to handle than that of algebraic cycles (cf. \cite{Dan20}). Very recently, Hu and Xiao \cite{HuXiao} give a purely combinatorial proof of Theorem~\ref{thm:main}.


As a consequence, we can get a B\'{e}zout-type inequality which was proved by Xiao \cite[Theorem~1.6]{Xia19} when the base field is $\mathbb{C}$, generalizing the
classical B\'{e}zout theorem for hypersurfaces.
Here we get a better constant than \cite[Theorem~1.6]{Xia19}.
\begin{cor}[{cf. \cite[Theorem~1.6]{Xia19}}]\label{cor:bezout} 
Let $X$ be a projective variety with an ample divisor
$H$ as polarization and let $A_1, \dots, A_r$ be nef divisor classes on $X$. Assume that $a_1, \dots, a_r\in \mathbb{Z}_{>0}$ and $|a| := \sum_{i=1}^r a_i \leq n$.
Let $Y_1,\dots,Y_r$ be subvarieties of cycle classes $A_1^{a_1}, \dots, A_r^{a_r}$, and assume that they have proper
intersection, then
$$
\deg_H(Y_1\cap\dots \cap Y_r)\leq \min_k\left\{\frac{\prod_{i=1}^r\binom{|a|}{a_i}}{\binom{|a|}{a_k}(H^n)^{r-1}}\right\}{\prod_{i=1}^r\deg_H(Y_i)}.
$$
\end{cor}

There is another application on degrees of the iterations of rational self-maps on projective varieties. Let
$f: X \dashrightarrow X$ be a dominant rational self-map of a projective variety $X$ of dimension $n$ defined over an algebraically closed field of arbitrary characteristic and let $H$ be a nef and big divisor on $X$. For any integer $0\leq i\leq n$, the {\it $i$-th degree} of $f$ with respect to $H$ is defined by
$$\deg_{i, H} (f) = (\pi_1^* H^{n-i}\cdot \pi_2^*H^i),$$
where $\pi_1$ and $\pi_2$ are the projections from the normalization of the graph of $f$ in $X \times X$ onto the first and the second factor, respectively.
In \cite{Dan20}, Dang proved a weaker form of \eqref{eq:main} in order to study the sequence of intermediate degrees of the iterations of a dominant
rational self-map and to recover the results in \cite{BFJ08, DS04, Tru20}. As an application of Theorem~\ref{thm:main}, we give a new and simple proof of 
\cite[Theorem~1]{Dan20} with better constants.
\begin{cor}[{cf. \cite[Theorem~1]{Dan20}}]\label{cor:dynamics}
Let $X$ be a projective variety of dimension $n$ and let $H$ be a nef and
big divisor on $X$. Fix an integer $0\leq i\leq n$.
\begin{enumerate}
 \item For any dominant rational self-maps $f, g$
on $X$, 
$$
\deg_{i, H} (f \circ g)\leq \frac{\binom{n}{i}}{(H^n)}\deg_{i, H} (f)\cdot \deg_{i, H} (g).
$$

\item For any nef and big divisor $L$ on $X$ and any dominant rational self-maps $f$
on $X$, $$
\deg_{i, H} (f)\leq \frac{\binom{n}{i}^2(H^{n-i}\cdot L^i)\cdot (L^{n-i}\cdot H^i)}{(L^n)^2}\deg_{i, L} (f).
$$
\end{enumerate}
\end{cor}
This result is essential in the definition of the {\it dynamical degree} of a self-rational map (cf. \cite{BFJ08, Dan20, DS04, Tru20}).

This paper is organized as the following.
In Section~\ref{sec 2}, we introduce definitions and basic knowledge on (multipoint) Okounkov bodies.
In Section~\ref{sec 3}, we develop properties of Okounkov bodies and prove Theorem~\ref{thm:main} along with 
several general statements. In Section~\ref{sec 4}, we give the applications of the main theorem. In Section~\ref{sec equality}, we give a characterization of the equality case of Theorem~\ref{thm:main} on surfaces.

\section{Preliminaries}\label{sec 2}

\subsection{Notation and Conventions}
We work over an algebraically closed field of arbitrary characteristic.
We adopt the standard notation and definitions in \cite{Laz04, Laz'04, LM09}. 
A {\it variety} is reduced and irreducible.
A {\it divisor} on a projective variety always means a Cartier divisor. When the base field is uncountable, a property holds for a {\it very general} choice of data if it is satisfied away from
a countable union of proper closed subvarieties of the relevant parameter space.

\subsection{Volumes}

Let $X$ be a projective variety of dimension $n$ and let $D$ be a divisor on $X$. The {\it volume} of $D$ is the real number
$$
\vol_X(D)=\lim_{m\to \infty}\frac{h^0(X,\OO_X(mD))}{m^n/n!}.
$$
We say that $D$ is {\it big} if $\vol_X( D)>0$.
For more details and properties of volumes, we refer to \cite[2.2.C]{Laz04} and \cite[11.4.A]{Laz'04}. By the homogeneity of volumes, this definition can be extended to $\mathbb{Q}$-divisors. Note that if $D$ is a nef divisor, then $\vol_X(D)=(D^n)$.

\subsection{Restricted volumes}
We recall the notation in \cite[\S 2.4]{LM09}.
Let $V$ be a projective variety and let $D$ be a big divisor on $V$. The {\it augmented base locus} $\mathbf{B}_{+}(D) \subset V$ is defined to be
$\mathbf{B}_+(D) = \mathbf{B}(D-A)$ for any sufficiently small ample $\mathbb{Q}$-divisor $A$, where $\mathbf{B}(D-A)$ is the {\it stable base locus} of $D-A$.
Let $X$ be an irreducible closed subvariety of $V$ of dimension $n$, then the {\it restricted volume} of $D$
from $V$ to $X$ is defined by
$$
\vol_{V|X}(D)=\lim_{m\to \infty}\frac{\dim\left(\text{\rm Im}\left(H^0(V, mD)\stackrel{\text{restr}}{\longrightarrow} H^0(X, mD|_X)\right)\right)}{m^n/n!}.
$$

For a sufficiently divisible integer $m>0$, consider $\pi_m: V_m\to V$ to be the blowing-up of $V$ along the base ideal of $|mD|$, then we have
$$
\pi_m^*|mD|=|M_m|+E_m
$$
where $M_m$ is free, and $E_m$ is the fixed part.
If $X\not\subset \mathbf{B}(D)$, the {\it asymptotic intersection number} of $D$ and $X$ is defined to be
$$
\|D^n\cdot X\|=\limsup_{m\to \infty}\frac{(M_m^n\cdot X_m)}{m^n},
$$
where $X_m$ is the strict transform of $X$ on $V_m$ (cf. \cite[Definition~2.6]{ELMNP09}).
Recall that we have the following Fujita's approximation theorem of restricted volumes.
\begin{thm}[{\cite[Theorem~2.13]{ELMNP09}, \cite[Remark~3.6]{LM09}}]\label{thm:fujita}
Let $V$ be a projective variety and let $D$ be a big divisor on $V$. Let $X$ be an irreducible closed subvariety of $V$ of dimension $n$ such that $X\not \subset \mathbf{B}_+(D)$.
Then $$\vol_{V|X}(D)=\|D^n\cdot X\|.$$
\end{thm}

\subsection{Okounkov bodies}
Okounkov bodies of big divisors were introduced in \cite{LM09, KK12} motivated by earlier works of Okounkov \cite{O1, O2}. There have been many interesting applications of Okounkov bodies in the study of geometric properties of divisors, for example, \cite{CHPW18, CHPW18b, CJPW19, CPW17b,CPW17a, Ito13,Jow10, KL17a, KL17, Roe16}.

We recall the definition of Okounkov bodies from \cite{LM09}.
Let $X$ be a projective variety of dimension $n$. Consider an \emph{admissible flag} $X_\bullet$ on $X$
$$
X_\bullet: X=X_0 \supseteq X_1\supseteq \dots \supseteq X_{n-1} \supseteq X_{n} = \{x \}
$$
where each $X_i$ is an irreducible closed subvariety of $X$ which is non-singular at the point $x$ and $\codim X_i=i$.
For a big divisor $D$ on $X$, we consider the $\mathbb{Q}$-linear system 
$$
|D|_{\mathbb{Q}}=\{D'\mid D\sim_\mathbb{Q} D'\geq 0\}
$$
and a valuation-like function
\begin{align*}
 \nu_{X_\bullet} : {}&|D|_{\mathbb{Q}} \to {\mathbb{R}}^{n}_{\geq 0}, \\
 {}&D' \mapsto \nu_{X_\bullet}(D')=(\nu_1, \nu_2, \ldots, \nu_n),
\end{align*}
where $\nu_i$ are defined inductively as follows:
\begin{enumerate}
 \item define $\nu_1:=\mult_{X_1}D'\in \mathbb{Q}_{\geq 0}$ and $D'_1:=D'-\nu_1X_1$ on $X_0$, and inductively,
 \item assuming that $\nu_i\in \mathbb{Q}_{\geq 0}$ and $D'_i$ on $X_{i-1}$ are defined, then define $\nu_{i+1}:=\mult_{X_{i+1}}(D'_i|_{X_i})\in \mathbb{Q}_{\geq 0}$ and $D'_{i+1}=D'_i|_{X_i}-\nu_{i+1}X_{i+1}$ on $X_i$.
\end{enumerate}
Here we remark that as $X_{i-1}$ is non-singular at $x$, $D'_i|_{X_i}$ is a well-defined $\mathbb{Q}$-divisor in a neighborhood of $x$ and $\mult_{X_{i+1}}(D'_i|_{X_i})$ can be well-defined.
The \emph{Okounkov body} of $D$ with respect to $X_\bullet$ is defined as
$$
\Delta_{X_\bullet}(D):=\text{the convex closure of }\nu_{X_\bullet}(|D|_{{\mathbb{Q}}}) \text{ in } {\mathbb{R}}^{n}_{\geq 0}.
$$
Here the definition is equivalent to \cite[Definition~1.8]{LM09} but we use $\mathbb{Q}$-linear systems instead of global sections as in \cite{LM09} to make the notation simpler. Such a formulation appears for example in \cite{CHPW18}.
By \cite[Theorem~A]{LM09}, we have
\begin{align}
 \vol_{{\mathbb{R}}^n}(\Delta_{X_\bullet}(D)) = \frac{1}{n!}\vol_X(D) \label{eq:vol1} 
\end{align}
for every admissible flag $X_\bullet$ on $X$.

\subsection{Multipoint Okounkov bodies}
In this paper, we will study the property of the Okounkov body of a given admissible flag cutting by very general hyperplanes (see Proposition~\ref{prop:former half}). For example, given an admissible flag $X_{\bullet}$ and a very general hyperplane $H$ on $X$, in order to naturally define an admissible flag $H_{\bullet}$, we need to pick a closed point in $X_{n-1}\cap H$. But in this way we lost the information of other points in $X_{n-1}\cap H$. So the natural idea is to consider the whole set of points $X_{n-1}\cap H$ instead of just picking one. To this end, we need to extend the definition of admissible flags and Okounkov bodies to multipoint admissible flags and multipoint Okounkov bodies.

Now we recall the definition of multipoint Okounkov bodies in \cite{Tru18}. For our purpose, we only introduce a special case.
Let $Z$ be a projective variety of dimension $k$. Consider a \emph{multipoint admissible flag} $Z_\bullet$ on $Z$
$$
Z_\bullet: Z=Z_0 \supseteq Z_1\supseteq \dots \supseteq Z_{k-1} \supseteq Z_{k} = \{p_1, \dots, p_N \}
$$
where $Z_k$ consists of $N$ distinct points and for $0\leq i<k$, each $Z_i$ is an irreducible closed subvariety of $Z$ which is non-singular at the points $p_1, \dots, p_N$ and $\codim Z_i=i$.
For $1\leq j\leq N$, denote by $Z_\bullet(p_j)$ the admissible flag 
$$
 Z=Z_0 \supseteq Z_1\supseteq \dots \supseteq Z_{k-1} \supseteq \{p_j \}.
$$
Then for a big divisor $D$ on $Z$, we can consider the functions $\nu_{Z_\bullet(p_j)}(D') \, (1\leq j\leq N)$ for $D'\in |D|_{\mathbb{Q}}$. Note that
$\nu_{Z_\bullet(p_j)}(D')$ only differs on the last coordinate in $\mathbb{R}^k$. So the lexicographical order of $\{\nu_{Z_\bullet(p_j)}(D')\mid 1\leq j\leq N\}$ is just the order of the last coordinates.
We define the subset $V_j(D)\subset |D|_\mathbb{Q}$ by
\begin{equation}\label{eq:Vj}
 V_j(D)=\left\{D'\in |D|_\mathbb{Q} ~\big |~ \nu_{Z_\bullet(p_j)}(D')< \nu_{Z_\bullet(p_i)}(D')\text{ for all } i\neq j\right\}. 
\end{equation}
The \emph{multipoint Okounkov body} of $D$ with respect to $Z_\bullet$ and $p_j$ is defined in \cite[Definition~3.4]{Tru18} as
$$
\Delta_{Z_\bullet, j}(D):=\text{the convex closure of }\nu_{Z_\bullet(p_j)}(V_j(D)) \text{ in } {\mathbb{R}}^{k}_{\geq 0}.
$$
By \cite[Theorem~A]{Tru18}, we have
\begin{align}
\sum_{j=1}^N\vol_{{\mathbb{R}}^k}(\Delta_{Z_\bullet, j}(D)) = \frac{1}{k!}\vol_Z(D) \label{eq:vol2}
\end{align}
for every multipoint admissible flag $Z_\bullet$ on $Z$.

By the homogeneity (cf. \cite[Proposition~3.10]{Tru18}), the definition of (multipoint) Okounkov bodies can be extended to big $\mathbb{Q}$-divisors. In particular, \eqref{eq:vol1} and \eqref{eq:vol2} hold for big $\mathbb{Q}$-divisors.

\begin{remark}
In \cite{Tru18}, it has been assumed that the base field is $\mathbb{C}$ and $Z$ is smooth. But as the proof of \cite[Theorem~A]{Tru18} follows the line of \cite[Theorem~A]{LM09}, it is not hard to check that \cite[Theorem~A]{Tru18} holds for any projective variety over any algebraically closed field. 
Here as in \cite[Remark~3.7]{LM09}, we do not need to assume that the base field is uncountable.
\end{remark}

\subsection{Cutting a divisor by general hyperplanes}

We will use the following lemma which is a direct consequence of Bertini's theorem.
\begin{lem}\label{lem:hypercut}
Let $X$ be a projective variety. 
Let $D$ be a $\mathbb{Q}$-divisor on $X$ and let $P$ be a prime divisor whose generic point lies in the smooth locus of $X$.
Then for a general very ample divisor $H$ on $X$, the following statements hold:
\begin{enumerate}
\item if $\dim X>2$, then $P|_H$ is a prime divisor whose generic point lies in the smooth locus of $H$;
\item if $\dim X=2$, then $P|_H$ is a reduced divisor (consisting of points) lies in the smooth locus of $H$;
\item $\mult_P D=\mult_{P'} D|_H$ for each irreducible component $P'$ of $P|_H$.
\end{enumerate}
\end{lem}

\begin{proof}
(1) and (2) follows from Bertini's theorem \cite[Theorem~II.8.18]{GTM52}. To get (3), we just need to choose $H$ general so that 
\begin{itemize}
\item for each irreducible component $P_1$ of $D$, $P_1|_H$ is reduced, and
\item for irreducible components $P_1$ and $P_2$ of $D$, $P_1|_H$ and $P_2|_H$ has no common irreducible component.
\end{itemize}
This is again by Bertini's theorem.
\end{proof}

\section{Inequalities between volumes via Okounkov bodies}
\label{sec 3}
\subsection{A comparision result on Okounkov bodies}\label{subsec:compare}
We define $\pr_{>k}:\mathbb{R}^n\to\mathbb{R}^{n-k} $ to be the projection of the last $n-k$ coordinates and $\pr_{\leq k}:\mathbb{R}^n\to\mathbb{R}^{k} $ to be the projection of the first $k$ coordinates. 

The following lemma gives a comparison of the Okounkov bodies of a certain admissible flag and its sub-flag via the natural projections.

\begin{lem}\label{lem:latter half}
Let $X$ be a projective variety of dimension $n$, let $D$ be a big divisor on $X$, and let $X_{\bullet}$ be an admissible flag. Fix an integer $1\leq k\leq n$. Denote $Y=X_{k}$ and $Y_{\bullet}=X_{k+\bullet}$. Suppose that $D|_Y$ is big and $\mathcal{O}_{X_{i-1}}(X_{i})|_{Y}$ is a semiample line bundle on $Y$ for $1\leq i\leq k$. Then
$$
\pr_{>k}(\Delta_{X_{\bullet}}(D))\subset \Delta_{Y_{\bullet}}(D|_Y).
$$

\end{lem}

\begin{proof}
Fix any $D'\in |D|_{\mathbb{Q}}$.
Suppose that $\nu_{X_{\bullet}}(D')=(\nu_1,\dots, \nu_n)$. 
Recall that for $0\leq i\leq n-1$ we define by induction that
$\nu_{i+1}=\mult_{X_{i+1}}D'_i|_{X_i}$ and
$D'_{i+1}=D'_i|_{X_i}-\nu_{i+1}X_{i+1}$.
Then by definition
$\nu_{Y_{\bullet}}(D'_k|_{Y})=(\nu_{k+1},\dots, \nu_n)$. 
By construction,
$$
D'_k|_{Y}\sim_\mathbb{Q} D|_Y-\sum_{i=1}^k \nu_iX_i|_Y,
$$
where we view $X_i|_Y$ as a divisor defined by the line bundle $\mathcal{O}_{X_{i-1}}(X_{i})|_{Y}$.
So by our assumption, $D|_Y-D'_k|_{Y}$ is semiample. Therefore, we can find $\tilde D\in |D|_Y|_{\mathbb{Q}}$ such that 
$\nu_{Y_{\bullet}}(\tilde{D})=\nu_{Y_{\bullet}}(D'_k|_{Y})=(\nu_{k+1},\dots, \nu_n)$. 
This concludes the desired inclusion.
\end{proof}

The following proposition describes the behavior of the Okounkov body of an admissible flag when cutting by very general hyperplanes.

\begin{prop}\label{prop:former half}
Let $X$ be a projective variety of dimension $n$ over an uncountable algebraically closed field, let $D$ be a big divisor on $X$, and let $X_{\bullet}$ be an admissible flag. Fix an integer $1\leq k\leq n-1$. 
For very general very ample divisors $H_1, \dots, H_{n-k}$ on $X$, denote $Z=H_1\cap \dots \cap H_{n-k}$, then there is a natural multipoint admissible flag $Z_\bullet$ on $Z$ given by $Z_i=X_i\cap Z$ for $0\leq i\leq k-1$, and $Z_{k}=X_{k}\cap Z=\{p_1,\dots, p_N\}$. Then
$$
\pr_{\leq k}(\Delta_{X_{\bullet}}(D))\subset \bigcap_{m\in \mathbb{Z}_{>0}}\Delta_{Z_{\bullet}, j}\left(D|_Z+\frac{1}{m}A_Z\right)
$$
for all $1\leq j\leq N$ and any ample divisor $A_Z$ on $Z$.
\end{prop}

\begin{remark}
By \cite[Proof of Theorem~A]{Tru18},
if $\Delta_{Z_{\bullet}, j}(D|_Z)^{\circ}\neq \emptyset$, then $$\bigcap_{m\in \mathbb{Z}_{>0}}\Delta_{Z_{\bullet}, j}\left(D|_Z+\frac{1}{m}A_Z\right)=\Delta_{Z_{\bullet}, j}(D|_Z).$$
But we do not need this fact in this paper.
\end{remark}
\begin{proof}
By definition,
the set $\nu_{X_\bullet}(|D|_{{\mathbb{Q}}})$ lies in ${\mathbb{Q}}^{n}_{\geq 0}$, which is a countable set. So 
there exists a countable set $S\subset |D|_{\mathbb{Q}}$ such that $\Delta_{X_{\bullet}}(D)$ is the convex closure of $\{\nu_{X_{\bullet}}(D')\mid D'\in S\}$. 
It suffices to show that for all $D'\in S\subset |D|_{\mathbb{Q}}$ and for very general very ample divisors $H_1, \dots, H_{n-k}$ on $X$, we have 
\begin{align}
\pr_{\leq k}(\nu_{X_{\bullet}}(D'))\in \Delta_{Z_{\bullet}, j}\left(D|_Z+\frac{1}{m}A_Z\right)\label{eq:111}
\end{align} 
for all $m\in \mathbb{Z}_{>0}$.

In fact, it suffices to show that \eqref{eq:111} holds for a single $D'\in S$ and for very general very ample divisors $H_1, \dots, H_{n-k}$ on $X$. 
More precisely, for fixed very ample linear systems $\mathcal{L}_1, \dots, \mathcal{L}_{n-k}$ on $X$, if \eqref{eq:111} holds for every $D'\in S$ and every $H_i\in \mathcal{L}_i\setminus Z_{i, D'}$ where $Z_{i, D'}\subset \mathcal{L}_i$ is a countable union of proper closed subset (depending on $D'$) for each $i$, then \eqref{eq:111} holds for every $D'\in S$ and every $H_i\in \mathcal{L}_i\setminus \bigcup_{D'\in S}Z_{i, D'}$. 
Here $\bigcup_{D'\in S}Z_{i, D'}$ is again a countable union of proper closed subset of $\mathcal{L}_i$ as $S$ is countable.

Now we consider a fixed $D'\in S$.
Suppose that $\nu_{X_{\bullet}}(D')=(\nu_1,\dots, \nu_n)$.
Recall that for $0\leq i\leq n-1$ we define by induction that
$\nu_{i+1}=\mult_{X_{i+1}}D'_i|_{X_i}$ and
$D'_{i+1}=D'_i|_{X_i}-\nu_{i+1}X_{i+1}$.
By applying Lemma~\ref{lem:hypercut} inductively, by taking $H_1, \dots, H_{n-k}$ general, we get that 
\begin{itemize}
 \item $\nu_{i+1}=\mult_{Z_{i+1}}D'_i|_{Z_i}$ and
$D'_{i+1}|_{Z_i}=D'_i|_{Z_i}-\nu_{i+1}Z_{i+1}$ for $0\leq i\leq k-2$;
\item $\nu_{k}=\mult_{p_s}D'_{k-1}|_{Z_{k-1}}$ for all $1\leq s\leq N$.
\end{itemize}
As $A_Z$ is ample, we can find $A'_Z\in |A_Z|_{\mathbb{Q}}$ whose support is very ample which does not contain $p_j$ and $Z_{k-1}$ but contains $p_s$ for all $s\neq j$, in other words,
\begin{itemize}
 \item $\mult_{p_j}A'_Z=\mult_{Z_{k-1}}A'_Z=0$, and
 \item $\mult_{p_s}A'_Z=\mu_s>0$ for all $s\neq j$.
\end{itemize} 
Then from the construction, 
$$\nu_{Z_{\bullet}(p_j)}(D'|_Z+\frac{1}{m}A'_Z)=(\nu_1,\dots, \nu_k)$$
and 
$$\nu_{Z_{\bullet}(p_s)}(D'|_Z+\frac{1}{m}A'_Z)=(\nu_1,\dots, \nu_{k-1}, \nu_k+\frac{1}{m}\mu_s)$$
for $s\neq j$.
Then we have $D'|_Z+\frac{1}{m}A'_Z\in V_j(D|_Z+\frac{1}{m}A_Z)$ as defined in \eqref{eq:Vj} and hence $(\nu_1,\dots, \nu_k)\in \Delta_{Z_{\bullet}, j}(D|_Z+\frac{1}{m}A_Z)$.
\end{proof}

\subsection{The generalized reverse Khovanskii--Teissier inequality}
As the volume of Okounkov body can compute the volume of big divisors, the comparision result in \S\ref{subsec:compare} will yield a stronger version of the reverse Khovanskii--Teissier inequality as follows.
\begin{thm}\label{thm:1}
Let $X$ be a projective variety of dimension $n$ over an uncountable algebraically closed field. Fix an integer $1\leq k\leq n-1$. Let $D$ be a big divisor on $X$ and let $B_1, \dots, B_{k}, C_1, \dots, C_{n-k}$ be very general very ample divisors on $X$. Denote $Y=B_1\cap \dots \cap B_{k}$ and $Z=C_1\cap \dots \cap C_{n-k}$.
Then 
$$
\vol_Y(D|_Y)\cdot \vol_Z(D|_Z)\geq \frac{k!(n-k)!}{n!}\vol_X(D)\cdot (Y\cdot Z).
$$
\end{thm}

\begin{proof}
It is easy to construct an admissible flag $X_{\bullet}$ such that $X_{i}=B_1\cap \dots \cap B_{i}$ for any $1\leq i\leq k$.
By applying Lemma~\ref{lem:latter half}, we get 
$$
\pr_{>k}(\Delta_{X_{\bullet}}(D))\subset \Delta_{Y_{\bullet}}(D|_Y).
$$
By applying Proposition~\ref{prop:former half} to the case $H_i=C_i$ for $1\leq i\leq n-k$, 
we get 
$$
\pr_{\leq k}(\Delta_{X_{\bullet}}(D))\subset \bigcap_{m\in \mathbb{Z}_{>0}}\Delta_{Z_{\bullet}, j}\left(D|_Z+\frac{1}{m}A_Z\right)
$$
for all $1\leq j\leq N$, where $Z=C_1\cap\dots \cap C_{n-k}$, $N=(X_k\cdot Z)=(Y\cdot Z)$, and $A_Z$ is an ample divisor on $Z$.
So
$$
\Delta_{X_{\bullet}}(D)\subset \Delta_{Z_{\bullet}, j}\left(D|_Z+\frac{1}{m}A_Z\right)\times \Delta_{Y_{\bullet}}(D|_Y)\subset \mathbb{R}^{k}\times \mathbb{R}^{n-k}
$$
for all $1\leq j\leq N$ and all $m\in\mathbb{Z}_{>0}$.
Combining with \eqref{eq:vol1} and \eqref{eq:vol2}, this yields
\begin{align*}
 {} \frac{1}{n!}(Y\cdot Z)\cdot \vol_X(D)= {}&N \vol_{\mathbb{R}^n}(\Delta_{X_{\bullet}}(D))\\
 \leq {}& \sum_{j=1}^N \vol_{\mathbb{R}^k}\left(\Delta_{Z_{\bullet}, j}\left(D|_Z+\frac{1}{m}A_Z\right)\right)\cdot \vol_{\mathbb{R}^{n-k}}(\Delta_{Y_{\bullet}}(D|_Y))
 \\
 ={}& \frac{1}{k!}\vol_Z\left(D|_Z+\frac{1}{m}A_Z\right)\cdot \frac{1}{(n-k)!}\vol_Y(D|_Y).
\end{align*}
We can conclude the assertion by taking $m\to \infty$, as $\vol_Z$ is a continuous function for big $\mathbb{Q}$-divisors by \cite[Corollary~4.12]{LM09}.
\end{proof}

As a consequence, we can obtain a general form of Theorem~\ref{thm:main}.

\begin{thm}\label{thm:2}
Let $X$ be a projective variety of dimension $n$. Fix an integer $1\leq k\leq n-1$.
Let $A, B_1, \dots, B_{k}, C_1, \dots, C_{n-k}$ be nef divisors on $X$. Then 
\begin{align}
 {}&(B_1\cdot\dots \cdot B_k\cdot A^{n-k})\cdot (A^k\cdot C_1\cdot\dots \cdot C_{n-k})\notag\\
 \geq {}&\frac{k!(n-k)!}{n!}(A^n)\cdot (B_1\cdot\dots \cdot B_k\cdot C_1\cdot\dots \cdot C_{n-k}).\label{eq:B1Bk}
\end{align}
Moreover, if $A, B_1, \dots, B_{k}, C_1, \dots, C_{n-k}$ are ample, then this inequality is strict.
\end{thm}

\begin{proof}
After base change we may assume that the base field is uncountable. 
As the inequality is homogeneous and nef divisors are the limits of ample $\mathbb{Q}$-divisors in the N\'eron--Severi group $\mathrm{NS}(X)\otimes_\mathbb{Z} \mathbb{R}$, it suffices to prove the inequality \eqref{eq:B1Bk} for very ample divisors $A, B_1, \dots, B_{k}, C_1, \dots, C_{n-k}$ on $X$.
Then the inequality \eqref{eq:B1Bk} follows from Theorem~\ref{thm:1} by taking $D=A$.

For the last statement, suppose that $A, B_1, \dots, B_{k}, C_1, \dots, C_{n-k}$ are ample. Without loss of generality, we may assume that $A$ is very ample and is a subvariety of $X$. Assume to the contrary that the equality holds. Note that for any sufficiently small $t>0$, $B_k-tA$ is ample. So applying \eqref{eq:B1Bk} to $A, B_1, \dots, B_{k-1}, B_{k}-tA, C_1, \dots, C_{n-k}$, we get 
\begin{align*}
 {}&(B_1\cdot\dots \cdot B_{k-1} \cdot (B_k-tA)\cdot A^{n-k})\cdot (A^k\cdot C_1\cdot\dots \cdot C_{n-k})\\
 \geq {}&\frac{k!(n-k)!}{n!}(A^n)\cdot (B_1\cdot\dots \cdot B_{k-1} \cdot (B_k-tA)\cdot C_1\cdot\dots \cdot C_{n-k}).
\end{align*}
By the assumption, this implies that 
\begin{align}
 {}&(B_1\cdot\dots \cdot B_{k-1} \cdot A^{n-k+1})\cdot (A^k\cdot C_1\cdot\dots \cdot C_{n-k})\notag\\
 \leq {}&\frac{k!(n-k)!}{n!}(A^n)\cdot (B_1\cdot\dots \cdot B_{k-1} \cdot A\cdot C_1\cdot\dots \cdot C_{n-k}).\label{eq:last1}
\end{align}
On the other hand, applying \eqref{eq:B1Bk} to $A|_A, B_1|_A, \dots, B_{k-1}|_A, C_1|_A, \dots, C_{n-k}|_A$ on the variety $A$, we get 
\begin{align}
 {}&(B_1\cdot\dots \cdot B_{k-1}\cdot A^{n-k+1})\cdot (A^k\cdot C_1\cdot\dots \cdot C_{n-k})\notag\\
 \geq {}&\frac{(k-1)!(n-k)!}{(n-1)!}(A^n)\cdot (B_1\cdot\dots \cdot B_{k-1} \cdot A\cdot C_1\cdot\dots \cdot C_{n-k}).\label{eq:last2}
\end{align}
Here we remark that we need $k>1$ in order to apply \eqref{eq:B1Bk}, but inequality \eqref{eq:last2} holds trivially if $k=1$.
Then \eqref{eq:last1} and \eqref{eq:last2} yields 
$$
(A^n)\cdot (B_1\cdot\dots \cdot B_{k-1} \cdot A\cdot C_1\cdot\dots \cdot C_{n-k})\leq 0,
$$
a contradiction.
\end{proof}
\begin{remark}
From the proof of Theorem~\ref{thm:2}, it is easy to see that the inequality \eqref{eq:B1Bk} is strict as long as $A, B_k$ are ample and $(B_1\cdot\dots \cdot B_{k-1} \cdot A\cdot C_1\cdot\dots \cdot C_{n-k})>0$.
\end{remark}

\begin{remark}
Although we deduce Theorem~\ref{thm:2} from Theorem~\ref{thm:1}, it is not hard to see that Theorem~\ref{thm:2} implies Theorem~\ref{thm:1} conversely, as the volume of a big divisor can be approximated by ample divisors by Fujita’s approximation theorem (cf. \cite[\S 3.1]{LM09} or Theorem~\ref{thm:fujita}). This fact was pointed out by Jian Xiao.
\end{remark}

\subsection{Proof of Theorem~\ref{thm:main}} It is just a special case of Theorem~\ref{thm:2}.

\subsection{Further discussions on restricted volumes}

In this subsection we will discuss the generalization of Theorem~\ref{thm:1} to restricted volumes, which was suggested by Jian Xiao. 

\begin{thm}\label{thm:1restricted}
Let $V$ be a projective variety over an uncountable algebraically closed field and let $D$ be a big divisor on $V$. Let $X$ be an irreducible closed subvariety of $V$ of dimension $n$ such that $X\not \subset \mathbf{B}_+(D)$.
 Fix an integer $1\leq k\leq n-1$. 
 Let $B_1, \dots, B_{k}, C_1, \dots, C_{n-k}$ be very general very ample divisors on $V$. Denote $Y=B_1\cap \dots \cap B_{k}\cap X$ and $Z=C_1\cap \dots \cap C_{n-k}\cap X$.
Then 
$$
\vol_{V|Y}(D)\cdot \vol_{V|Z}(D)\geq \frac{k!(n-k)!}{n!}\vol_{V|X}(D)\cdot (Y\cdot Z)_X.
$$
\end{thm}
\begin{proof}
For any sufficiently divisible integer $m>0$, consider $\pi_m: V_m\to V$ to be the blowing-up of $V$ along the base ideal of $|mD|$, then we have
$$
\pi_m^*|mD|=|M_m|+E_m
$$
where $M_m$ is free, and $E_m$ is the fixed part.
Denote by $X_m$ the strict transform of $X$ on $V_m$. By taking 
$B_1, \dots, B_{k}, C_1, \dots, C_{n-k}$ very general, we may assume that $Y, Z\not \subset \mathbf{B}_+(D)$ and for every sufficiently divisible integer $m>0$, $Y_m=(\pi_m^{*}B_1\cdot \dots \cdot \pi_m^{*}B_k\cdot X_m)$ and $Z_m=(\pi_m^{*}C_1\cdot \dots \cdot \pi_m^{*}C_{n-k}\cdot X_m)$ as cycles, where $Y_m, Z_m$ are the strict transforms of $Y, Z$ on $V_m$ respectively.
Then by applying Theorem~\ref{thm:2} to $$M_m|_{X_m}, \pi^*_mB_1|_{X_m}, \dots, \pi^*_mB_{k}|_{X_m}, \pi^*_mC_1|_{X_m}, \dots, \pi^*_mC_{n-k}|_{X_m}$$ on $X_m$, we get 
\begin{align*}
 {}&( \pi^*_m B_1\cdot\dots \cdot \pi^*_m B_k\cdot M_m^{n-k}\cdot X_m)\cdot (M_m^k\cdot \pi^*_mC_1\cdot\dots \cdot \pi^*_mC_{n-k}\cdot X_m)\\
 \geq {}&\frac{k!(n-k)!}{n!}(M_m^n\cdot X_m)\cdot (\pi^*_mB_1\cdot\dots \cdot \pi^*_mB_k\cdot \pi^*_mC_1\cdot\dots \cdot \pi^*_mC_{n-k}\cdot X_m).
\end{align*}
By taking $m\to \infty$, we get
$$
\|D^k\cdot Y\|\cdot \|D^{n-k}\cdot Z\|\geq \frac{k!(n-k)!}{n!}\|D^{n}\cdot X\|\cdot (Y\cdot Z)_X.
$$
This concludes the desired inequality by Theorem~\ref{thm:fujita}.
\end{proof}

\section{Applications}\label{sec 4}
\subsection{Proof of Corollary~\ref{cor:bezout}}
It is equivalent to showing that
$$
(A_1^{a_1}\cdot \cdots \cdot A_r^{a_r}\cdot H^{n-|a|})\leq \min_k\left\{\frac{\prod_{i=1}^r\binom{|a|}{a_i}}{\binom{|a|}{a_k}(H^n)^{r-1}}\right\}{\prod_{i=1}^r(A_i^{a_i}\cdot H^{n-a_i})}.
$$
Without loss of generality, it suffices to show that 
\begin{align}
 (H^n)^{r-1}\cdot (A_1^{a_1}\cdot \cdots \cdot A_r^{a_r}\cdot H^{n-|a|})\leq {\prod_{i=2}^r\binom{|a|}{a_i}}{\prod_{i=1}^r(A_i^{a_i}\cdot H^{n-a_i})}.\label{eq:bezout last}
\end{align}
Without loss of generality, we may assume that $H$ is very ample. 
When $r=1$, \eqref{eq:bezout last} is trivial. Suppose that $r\geq 2$.
Take general elements $H_1,\dots, H_{n-|a|}\in |H|$ and take $Z=H_1\cap\dots\cap H_{n-|a|}$.
Applying Theorem~\ref{thm:2} to $A_r|_Z$ and $H|_Z$, we get
\begin{align*}
 {}&(H^n)\cdot (A_1^{a_1}\cdot \cdots \cdot A_r^{a_r}\cdot H^{n-|a|})\\
 ={}& (H^{|a|}|_Z)\cdot (A_1^{a_1}|_Z\cdot \cdots \cdot A_r^{a_r}|_Z)\\
 \leq{}&\binom{|a|}{a_r} (A_1^{a_1}|_Z\cdot \cdots \cdot A_{r-1}^{a_{r-1}}|_Z\cdot H^{a_r}|_Z)\cdot ( A_r^{a_r}|_Z\cdot H^{|a|-a_r}|_Z)\\
 ={}&\binom{|a|}{a_r} (A_1^{a_1}\cdot \cdots \cdot A_{r-1}^{a_{r-1}}\cdot H^{n-|a|+a_r})\cdot(A_r^{a_r}\cdot H^{n-a_r}).
\end{align*}
So we can prove inequality \eqref{eq:bezout last} by induction on $r$. 

\begin{remark}
The constant we obtain is slightly better than Xiao's, which is $\min_k\left\{\frac{\prod_{i=1}^r\binom{n}{a_i}}{\binom{n}{a_k}(H^n)^{r-1}}\right\}$, but the proof is basically the same as long as one knows Theorem~\ref{thm:2}. It remains interesting to find the optimal constant. We refer to \cite[\S 3.4]{Xia15} for some related discussions.
\end{remark}

\subsection{Proof of Corollary~\ref{cor:dynamics}}
(1) For any dominant rational self-maps $f, g$
on $X$, take a projective normal variety $W$ with generically finite morphisms $p_j: W\to X$ for $j=1, 2, 3$ such that $p_2=p_1\circ g$ and $p_3=p_2\circ f$. 
Then by the projection formula between $W$ and the graphs of $f, g$, 
\begin{enumerate}
 \item $\deg_{i, H} (f \circ g)=\frac{1}{\deg(p_1)}\cdot (p_1^*H^{n-i}\cdot p_3^*H^{i})$;
 \item $\deg_{i, H} (f)=\frac{1}{\deg(p_2)}\cdot (p_2^*H^{n-i}\cdot p_3^*H^{i})$;
 \item $\deg_{i, H} (g)=\frac{1}{\deg(p_1)}\cdot (p_1^*H^{n-i}\cdot p_2^*H^{i})$.
\end{enumerate}
If $i=0$ or $n$ the statement is trivially true. If $1\leq i\leq n-1$, by Theorem~\ref{thm:main}, we have
\begin{align*}
 \deg_{i, H} (f \circ g)= {}&\frac{1}{\deg(p_1)}\cdot (p_1^*H^{n-i}\cdot p_3^*H^{i})\\
 \leq {}& \frac{1}{\deg(p_1)}\cdot\frac{\binom{n}{i}}{(p_2^*H^n)} (p_1^*H^{n-i}\cdot p_2^*H^{i})\cdot (p_2^*H^{n-i}\cdot p_3^*H^{i}) \\
 ={}& \frac{\binom{n}{i}}{(H^n)}\deg_{i, H} (f)\cdot \deg_{i, H} (g).
\end{align*}

(2) Let $\pi_1$ and $\pi_2$ be the projections from the normalization of the graph of $f$ in $X \times X$ onto the first and the second factor, respectively.
If $i=0$ or $n$ the statement is trivially true. If $1\leq i\leq n-1$ then by applying Theorem~\ref{thm:main} repeatedly, one can get
\begin{align*}
 \deg_{i, H} (f)={}& (\pi_1^* H^{n-i}\cdot \pi_2^*H^i)\\
 \leq {}&\frac{\binom{n}{i}}{(\pi_1^*L^n)}(\pi_1^* H^{n-i}\cdot \pi_1^*L^i)\cdot (\pi_1^* L^{n-i}\cdot \pi_2^*H^i)\\
 \leq {}&\frac{\binom{n}{i}}{(\pi_1^*L^n)}(\pi_1^* H^{n-i}\cdot \pi_1^*L^i)\cdot \frac{\binom{n}{i}}{(\pi_2^*L^n)}(\pi_1^* L^{n-i}\cdot \pi_2^*L^i)\cdot (\pi_2^*L^{n-i}\cdot \pi_2^*H^i)\\
 = {}&\frac{\binom{n}{i}^2(H^{n-i}\cdot L^i)\cdot (L^{n-i}\cdot H^i)}{ (L^n)^2}\deg_{i, L} (f).
\end{align*}

\begin{remark}
As we apply the optimal inequality \eqref{eq:main}, we obtain better constants than Dang's in \cite[Theorem~1]{Dan20}. For example, the constant in \cite[Theorem~1(i)]{Dan20} is $\frac{(n-i+1)^i}{(H^n)}$. It remains interesting to find the optimal constants.
\end{remark}

\section{Characterization of the equality case on surfaces}\label{sec equality}

In this section, we give a characterization of the equality case of Theorem~\ref{thm:main} on surfaces. The proof is motivated by \cite[Proposition~1.15]{BFJ08}. For simplicity, we state the characterization for non-singular surfaces, and the general case can be easily worked out by taking a desingularization.
\begin{prop}\label{prop:equality}
Let $X$ be a non-singular projective surface.
Let $A, B, C$ be nef divisors on $X$. Then 
\begin{align*}
 2(B\cdot A)\cdot (A\cdot C)= (A^2)\cdot (B\cdot C)\neq 0
\end{align*}
if and only if the following conditions hold
\begin{enumerate}
 \item $A\equiv s B+t C$ in $\mathrm{NS}(X)\otimes_\mathbb{Z} \mathbb{R}$ for some $s, t>0$;
 \item $B^2=C^2=0$, $(B\cdot C)\neq 0$.
\end{enumerate}
\end{prop}

\begin{proof}
The ``if" part is obvious. We only deal with the ``only if" part.

Suppose that $
2(B\cdot A)\cdot (A\cdot C)= (A^2)\cdot (B\cdot C)\neq 0
$.
Then $(B\cdot A), (A\cdot C), (A^2), (B\cdot C)$ are all positive numbers.
Set $\Gamma=A-\frac{(A^2)}{2(A\cdot B)}B$.
Then $(\Gamma\cdot C)=0$ and $(\Gamma^2)=\frac{(A^2)^2(B^2)}{4(A\cdot B)^2}\geq 0$.

We claim that $\Gamma$ and $C$ are propotional in $\mathrm{NS}(X)\otimes_\mathbb{Z} \mathbb{R}$ by the Hodge index theorem. Fix any ample divisor $H$ on $X$, by our assumption, $C$ is not numerically trivial, so $(H\cdot C)\neq 0$. 
Then $((\Gamma-\frac{(\Gamma\cdot H)}{(C\cdot H)}C)\cdot H)=0.$
On the other hand, 
$((\Gamma-\frac{(\Gamma\cdot H)}{(C\cdot H)}C)^2)\geq 0.$
So the classical Hodge index theorem (see \cite[Theorem~V.1.9]{Har}) implies that $\Gamma-\frac{(\Gamma\cdot H)}{(C\cdot H)}C\equiv 0$. As they are propotional, we can also get $\Gamma\equiv \frac{(\Gamma\cdot A)}{(C\cdot A)}C$, which implies that
\begin{align}
 A\equiv \frac{(A^2)}{2(B\cdot A)}B+\frac{(A^2)}{2(C\cdot A)}C\label{ABC surface}
\end{align}
by the construction of $\Gamma$. The fact that $(B^2)=(C^2)=0$ can be obtained by intersecting \eqref{ABC surface} with $B, C$ respectively.
\end{proof}

\section*{Acknowledgments}
We are grateful to Jian Xiao for drawing our attention to this topic and many helpful comments.
We would like to thank Mingchen Xia for discussions on Okounkov bodies.
We would like to thank the referee for useful suggestions.

The authors were supported by NSFC for Innovative Research Groups (Grant No.~12121001) and partially supported by National Key Research and Development Program of China (Grant No.~2020YFA0713200). The second author was also supported by Shanghai Pilot Program for Basic Research (No. 21TQ00).
The authors are members of LMNS, Fudan University.
\bibliographystyle{plain}
\bibliography{main}

\begin{thebibliography}{10}

\bibitem{BFJ08}
S\'{e}bastien Boucksom, Charles Favre, and Mattias Jonsson.
\newblock Degree growth of meromorphic surface maps.
\newblock {\em Duke Math. J.}, 141(3):519--538, 2008.

\bibitem{CHPW18}
Sung~Rak Choi, Yoonsuk Hyun, Jinhyung Park, and Joonyeong Won.
\newblock Asymptotic base loci via {O}kounkov bodies.
\newblock {\em Adv. Math.}, 323:784--810, 2018.

\bibitem{CHPW18b}
Sung~Rak Choi, Yoonsuk Hyun, Jinhyung Park, and Joonyeong Won.
\newblock Okounkov bodies associated to pseudoeffective divisors.
\newblock {\em J. Lond. Math. Soc. (2)}, 97(2):170--195, 2018.

\bibitem{CJPW19}
Sung~Rak Choi, Seung-Jo Jung, Jinhyung Park, and Joonyeong Won.
\newblock A product formula for volumes of divisors via {O}kounkov bodies.
\newblock {\em Int. Math. Res. Not. IMRN}, (22):7118--7137, 2019.

\bibitem{CPW17b}
Sung~Rak Choi, Jinhyung Park, and Joonyeong Won.
\newblock Okounkov bodies and {Z}ariski decompositions on surfaces.
\newblock {\em Bull. Korean Math. Soc.}, 54(5):1677--1697, 2017.

\bibitem{CPW17a}
Sung~Rak Choi, Jinhyung Park, and Joonyeong Won.
\newblock Okounkov bodies associated to pseudoeffective divisors {II}.
\newblock {\em Taiwanese J. Math.}, 21(3):601--620, 2017.

\bibitem{Dan20}
Nguyen-Bac Dang.
\newblock Degrees of iterates of rational maps on normal projective varieties.
\newblock {\em Proc. Lond. Math. Soc. (3)}, 121(5):1268--1310, 2020.

\bibitem{DX21}
Nguyen-Bac Dang and Jian Xiao.
\newblock Positivity of valuations on convex bodies and invariant valuations by
  linear actions.
\newblock {\em J. Geom. Anal.}, 31(11):10718--10777, 2021.

\bibitem{DS04}
Tien-Cuong Dinh and Nessim Sibony.
\newblock Regularization of currents and entropy.
\newblock {\em Ann. Sci. \'{E}cole Norm. Sup. (4)}, 37(6):959--971, 2004.

\bibitem{ELMNP09}
Lawrence Ein, Robert Lazarsfeld, Mircea Musta\c{t}\u{a}, Michael Nakamaye, and
  Mihnea Popa.
\newblock Restricted volumes and base loci of linear series.
\newblock {\em Amer. J. Math.}, 131(3):607--651, 2009.

\bibitem{GTM52}
Robin Hartshorne.
\newblock {\em Algebraic geometry}.
\newblock Graduate Texts in Mathematics, No. 52. Springer-Verlag, New
  York-Heidelberg, 1977.

\bibitem{Har}
Robin Hartshorne.
\newblock {\em Algebraic geometry}.
\newblock Graduate Texts in Mathematics, No. 52. Springer-Verlag, New
  York-Heidelberg, 1977.

\bibitem{HuXiao}
Jiajun Hu and Jian Xiao.
\newblock Intersection theoretic inequalities via lorentzian polynomials.
\newblock {\em arXiv:2304.04191}, 2023.

\bibitem{Ito13}
Atsushi Ito.
\newblock Okounkov bodies and {S}eshadri constants.
\newblock {\em Adv. Math.}, 241:246--262, 2013.

\bibitem{Jow10}
Shin-Yao Jow.
\newblock Okounkov bodies and restricted volumes along very general curves.
\newblock {\em Adv. Math.}, 223(4):1356--1371, 2010.

\bibitem{KK12}
Kiumars Kaveh and A.~G. Khovanskii.
\newblock Newton-{O}kounkov bodies, semigroups of integral points, graded
  algebras and intersection theory.
\newblock {\em Ann. of Math. (2)}, 176(2):925--978, 2012.

\bibitem{KL17a}
Alex K\"{u}ronya and Victor Lozovanu.
\newblock Infinitesimal {N}ewton-{O}kounkov bodies and jet separation.
\newblock {\em Duke Math. J.}, 166(7):1349--1376, 2017.

\bibitem{KL17}
Alex K\"{u}ronya and Victor Lozovanu.
\newblock Positivity of line bundles and {N}ewton-{O}kounkov bodies.
\newblock {\em Doc. Math.}, 22:1285--1302, 2017.

\bibitem{Laz04}
Robert Lazarsfeld.
\newblock {\em Positivity in algebraic geometry. {I}}, volume~48 of {\em
  Ergebnisse der Mathematik und ihrer Grenzgebiete. 3. Folge. A Series of
  Modern Surveys in Mathematics [Results in Mathematics and Related Areas. 3rd
  Series. A Series of Modern Surveys in Mathematics]}.
\newblock Springer-Verlag, Berlin, 2004.
\newblock Classical setting: line bundles and linear series.

\bibitem{Laz'04}
Robert Lazarsfeld.
\newblock {\em Positivity in algebraic geometry. {II}}, volume~49 of {\em
  Ergebnisse der Mathematik und ihrer Grenzgebiete. 3. Folge. A Series of
  Modern Surveys in Mathematics [Results in Mathematics and Related Areas. 3rd
  Series. A Series of Modern Surveys in Mathematics]}.
\newblock Springer-Verlag, Berlin, 2004.
\newblock Positivity for vector bundles, and multiplier ideals.

\bibitem{LM09}
Robert Lazarsfeld and Mircea Musta\c{t}\u{a}.
\newblock Convex bodies associated to linear series.
\newblock {\em Ann. Sci. \'{E}c. Norm. Sup\'{e}r. (4)}, 42(5):783--835, 2009.

\bibitem{LX'15}
Brian Lehmann and Jian Xiao.
\newblock Convexity and {Z}ariski decomposition structure.
\newblock {\em arXiv:1507.04316v2}, 2015.

\bibitem{LX16}
Brian Lehmann and Jian Xiao.
\newblock Convexity and {Z}ariski decomposition structure.
\newblock {\em Geom. Funct. Anal.}, 26(4):1135--1189, 2016.

\bibitem{LX17}
Brian Lehmann and Jian Xiao.
\newblock Correspondences between convex geometry and complex geometry.
\newblock {\em \'{E}pijournal G\'{e}om. Alg\'{e}brique}, 1:Art. 6, 29, 2017.

\bibitem{O1}
Andrei Okounkov.
\newblock Brunn-{M}inkowski inequality for multiplicities.
\newblock {\em Invent. Math.}, 125(3):405--411, 1996.

\bibitem{O2}
Andrei Okounkov.
\newblock Why would multiplicities be log-concave?
\newblock In {\em The orbit method in geometry and physics ({M}arseille,
  2000)}, volume 213 of {\em Progr. Math.}, pages 329--347. Birkh\"{a}user
  Boston, Boston, MA, 2003.

\bibitem{Pop16}
Dan Popovici.
\newblock Sufficient bigness criterion for differences of two nef classes.
\newblock {\em Math. Ann.}, 364(1-2):649--655, 2016.

\bibitem{Roe16}
Joaquim Ro\'{e}.
\newblock Local positivity in terms of {N}ewton-{O}kounkov bodies.
\newblock {\em Adv. Math.}, 301:486--498, 2016.

\bibitem{Tru20}
Tuyen~Trung Truong.
\newblock Relative dynamical degrees of correspondences over a field of
  arbitrary characteristic.
\newblock {\em J. Reine Angew. Math.}, 758:139--182, 2020.

\bibitem{Tru18}
Antonio Trusiani.
\newblock Multipoint {O}kounkov bodies.
\newblock {\em arXiv:1804.02306v5, to appear in "Annales de l'Institut
  Fourier"}, 2018.

\bibitem{Xia15}
Jian Xiao.
\newblock Weak transcendental holomorphic {M}orse inequalities on compact
  {K}\"{a}hler manifolds.
\newblock {\em Ann. Inst. Fourier (Grenoble)}, 65(3):1367--1379, 2015.

\bibitem{Xia19}
Jian Xiao.
\newblock B\'{e}zout-type inequality in convex geometry.
\newblock {\em Int. Math. Res. Not. IMRN}, (16):4950--4965, 2019.

\end{thebibliography}

\end{document}